\documentclass{article}%
\usepackage{graphicx}
\usepackage{amsmath}
\usepackage{amsfonts}
\usepackage{amssymb}%
\newtheorem{theorem}{Theorem}

\newtheorem{corollary}[theorem]{Corollary}

\newtheorem{proposition}[theorem]{Proposition}

\newenvironment{proof}[1][Proof]{\textbf{#1.} }{\ \rule{0.5em}{0.5em}}
\begin{document}

\title{Curvature in Synthetic Differential Geometry of Groupoids}
\author{Hirokazu Nishimura\\Institute of Mathematics, University of Tsukuba\\Tsukuba, Ibaraki, 305-8571, Japan}
\maketitle

\begin{abstract}
We study the fundamental properties of curvature in groupoids within the
framework of synthetic differential geometry. As is usual in synthetic
differential geometry, its combinatorial nature is emphasized. In particular,
the classical Bianchi identity is deduced from its combinatorial one.

\end{abstract}

\section{Introduction}

The notion of curvature, which is one of the most fundamental concepts in
differential geometry, retrieves its combinatorial or geometric meaning in
synthetic differential geomety. It was Kock \cite{k4} who studied it up to the
second Bianchi identity synthetically for the first time. In particular, he
has revealed the combinatorial nature of the second Bianchi identity by
deducing it from an abstract one.

In \cite{k4} Kock trotted out first neighborhood relations, which are indeed
to be seen in formal manifolds, but which are no longer expected to be seen in
microlinear spaces in general. Since we believe that microlinear spaces should
play the same role in synthetic differential geometry as smooth manifolds have
done in classical differential geometry, we have elevated his ideas to a
microlinear context in \cite{n1}. However we were not so happy, because our
proof of the second Bianchi identity there appeared unnecessarily involved,
making us feel that we were somewhat off the point, though the proof was
completely correct.

Recently we got accustomed to groupoids, which encouraged us to attack the
same problem once again. Within the framework of groupoids, we find it
pleasant to think multiplicatively rather than additively (cf. Nishimura
\cite{n3}), which helps grasp the nature of the second Bianchi identity
firmly. Now we are to the point. What we have to do in order to deduce the
classical second Bianchi identity from the combinatorial one is only to note
some commutativity on the infinitesimal level, though groupoids are, by and
large, highly noncommutative. Our present experience is merely an example of
the familiar wisdom in mathematics that a good generalization reveals the nature.

\section{Preliminaries}

\subsection{Synthetic Differential Geometry}

Our standard reference on synthetic differential geometry is Chapters 1-5 of
Lavendhomme \cite{l1}. We will work internally within a good topos, in which
the intended set $\mathbb{R}$ of real numbers is endowed with a cornucopia of
nilpotent infinitesimals pursuant to the general Kock-Lawvere axiom. To see
how to build such a good topos, the reader is referred to Kock \cite{k1} or
Moerdijk and Reyes \cite{mr1}. Any space mentioned in this paper will be
assumed to be microlinear, unless stated to the contrary. We denote by $D$ the
set $\{d\in\mathbb{R}\mid d^{2}=0\}$, as is usual in synthetic differential geometry.

Given a group $G$, we denote by $\mathcal{A}G$ the tangent space of $G$ at its
identity, i.e., the totality of mappings $t:D\rightarrow G$ such that $t_{0}$
is the identity of $G$. We will often write $t_{d}$ rather than $t(d)$ for any
$d\in D$. As we will see shortly, $\mathcal{A}G$ is more than an $\mathbb{R}$-module.

\begin{proposition}
\label{t1.1}For any $t\in\mathcal{A}G$ and any $(d_{1},d_{2})\in D(2)$, we
have
\[
t_{d_{1}+d_{2}}=t_{d_{1}}t_{d_{2}}=t_{d_{2}}t_{d_{1}}%
\]
so that $t_{d_{1}}$ and $t_{d_{2}}$ commute.
\end{proposition}

\begin{proof}
By the same token as in Proposition 3 of \S 3.2 of Lavendhomme \cite{l1}.
\end{proof}

As an easy corollary of this proposition, we can see that
\[
t_{-d}=(t_{d})^{-1}%
\]
for we have $(d,-d)\in D(2)$.

\begin{proposition}
\label{t1.2}For any $t_{1},t_{2}\in\mathcal{A}G$, we have
\[
(t_{1}+t_{2})_{d}=(t_{2})_{d}(t_{1})_{d}=(t_{1})_{d}(t_{2})_{d}%
\]
for any $d\in D$, so that $(t_{1})_{d}$ and $(t_{2})_{d}$ commute.
\end{proposition}

\begin{proof}
By the same token as in Proposition 6 of \S 3.2 of Lavendhomme \cite{l1}.
\end{proof}

As an easy corollary of this proposition, we can see, by way of example, that
$(t_{1})_{d_{1}d_{2}}$ and $(t_{2})_{d_{1}d_{3}}$ commute for any $d_{1}%
,d_{2},d_{3}\in D$, for we have
\[
(t_{1})_{d_{1}d_{2}}(t_{2})_{d_{1}d_{3}}=(d_{2}t_{1})_{d_{1}}(d_{3}%
t_{2})_{d_{1}}=(d_{3}t_{2})_{d_{1}}(d_{2}t_{1})_{d_{1}}=(t_{2})_{d_{1}d_{3}%
}(t_{1})_{d_{1}d_{2}}%
\]

\begin{proposition}
\label{t1.3}For any $t_{1},t_{2}\in\mathcal{A}G$, there exists a unique
$s\in\mathcal{A}G$ such that
\[
s_{d_{1}d_{2}}=(t_{2})_{-d_{2}}(t_{1})_{-d_{1}}(t_{2})_{d_{2}}(t_{1})_{d_{1}}%
\]
for any $d_{1},d_{2}\in D$.
\end{proposition}

\begin{proof}
By the same token as in pp.71-72 of Lavendhomme \cite{l1}.
\end{proof}

We will write $[t_{1},t_{2}]$ for the above $s$.

\begin{theorem}
\label{t1.4}The $\mathbb{R}$-module $\mathcal{A}G$ endowed with the above Lie
bracket $[\cdot,\cdot]$ is a Lie algebra over $\mathbb{R}$.
\end{theorem}

\begin{proof}
By the same token as in our previous paper \cite{n2}.
\end{proof}

\subsection{Groupoids}

Groupoids are, roughly speaking, categories whose morphisms are always
invertible. Our standard reference on groupoids is MacKenzie \cite{ma1}. Given
a groupoid $G$ over a base $M$ with its object inclusion map $\mathrm{id}%
:M\rightarrow G$ and its source and target projections $\alpha,\beta
:G\rightarrow M$, we denote by $\mathfrak{B}(G)$ the totality of bisections of
$G $, i.e., the totality of mappings $\sigma:M\rightarrow G$ such that
$\alpha\circ\sigma$ is the identity mapping on $M$ and $\beta\circ\sigma$ is a
bijection of $M$ onto $M$. It is well known that $\mathfrak{B}(G)$ is a group
with respect to the operation $\ast$, where for any $\sigma,\rho
\in\mathfrak{B}(G)$, $\sigma\ast\rho\in\mathfrak{B}(G)$ is defined to be
\[
(\sigma\ast\rho)(x)=\sigma((\beta\circ\rho)(x))\rho(x)
\]
for any $x\in M$. It can easily be shown that the space $\mathfrak{B}(G)$ is
microlinear, provided that both $M$ and $G$ are microlinear, for which the
reader is referred to Proposition 6 of Nishimura \cite{n2}.

Given $x\in M$, we denote by $\mathcal{A}_{x}^{n}G$ the totality of mappings
$\gamma:D^{n}\rightarrow G$ with $\gamma(0,...,0)=\mathrm{id}_{x}$ and
$(\alpha\circ\gamma)(d_{1},...,d_{n})=x$ for any $(d_{1},...,d_{n})\in D^{n}$.
We denote by $\mathcal{A}^{n}G$ the set-theoretic union of $\mathcal{A}%
_{x}^{n}G$'s for all $x\in M$. In particular, we usually write $\mathcal{A}%
_{x}G$ and $\mathcal{A}G$ in place of $\mathcal{A}_{x}G$ and $\mathcal{A}G$
respectively. It is easy to see that $\mathcal{A}G$ is naturally a vector
bundle over $M$. A morphism $\varphi:H\rightarrow G$ of groupoids over $M$
naturally gives rise to a morphism $\varphi_{\ast}:\mathcal{A}H\rightarrow
\mathcal{A}G$ of vector bundles over $M$. As in \S 3.2.1 of Lavendhomme
\cite{l1}, where three distinct but equivalent viewpoints of vector fields are
presented, the totality $\Gamma(\mathcal{A}G)$ of sections of the vector
bundle $\mathcal{A}G$ can canonically be identified with the totality of
tangent vectors to $\mathfrak{B}(G)$ at $\mathrm{id}$, for which the reader is
referred to Nishimura \cite{n2}. We will enjoy this identification freely, and
we dare to write $\Gamma(\mathcal{A}G)$ for the totality of tangent vectors to
$\mathfrak{B}(G)$ at $\mathrm{id}$. Given $X,Y\in\Gamma(\mathcal{A}G)$, we
define a microsquare $Y\ast X$ to $\mathfrak{B}(G)$ at $\mathrm{id}$ to be
\[
(Y\ast X)(d_{1},d_{2})=Y_{d_{2}}\ast X_{d_{1}}%
\]
for any $(d_{1},d_{2})\in D^{2}$.

Given $\gamma\in\mathcal{A}^{n+1}G$ and $e\in D$, we define $\gamma_{e}^{i}%
\in\mathcal{A}^{n}G\mathcal{\ }$($1\leq i\leq n+1$) to be%

\[
\gamma_{e}^{i}(d_{1},...,d_{n})=\gamma(d_{1},...,d_{i-1},e,d_{i}%
,...,d_{n})\gamma(0,...,0,\underset{i}{e},0,...,0)^{-1}%
\]
for any $(d_{1},...,d_{n})\in D^{n}$. For our later use in the last section of
this paper, we introduce a variant of this notation. Given $\gamma
\in\mathcal{A}^{n+2}G$ and $e_{1},e_{2}\in D$, we define $\gamma_{e_{1},e_{2}%
}^{i,j}\in\mathcal{A}^{n}G\mathcal{\ }$($1\leq i<j\leq n+2$) to be
\[
\gamma_{e_{1},e_{2}}^{i,j}(d_{1},...,d_{n})=\gamma(d_{1},...,d_{i-1}%
,e_{1},d_{i},...,d_{j-2},e_{2},d_{j-1},...,d_{n})\gamma(0,...,0,\underset
{i}{e_{1}},...,\underset{j}{e_{2}},0,...,0)^{-1}%
\]
Given $\gamma\in\mathcal{A}^{2}G$, we define $\tau_{\gamma}^{1}\in
\mathcal{A}^{2}G$ to be
\[
\tau_{\gamma}^{1}(d_{1},d_{2})=\gamma(d_{1},0)
\]
for any $(d_{1},d_{2})\in D^{2}$. Similarly, given $\gamma\in\mathcal{A}^{2}%
G$, we define $\tau_{\gamma}^{2}\in\mathcal{A}^{2}G$ to be
\[
\tau_{\gamma}^{2}(d_{1},d_{2})=\gamma(0,d_{2})
\]
for any $(d_{1},d_{2})\in D^{2}$. Given $\gamma\in\mathcal{A}^{2}G$, we define
$\Sigma\gamma\in\mathcal{A}^{2}G$ to be%
\[
(\Sigma\gamma)(d_{1},d_{2})=\gamma(d_{2},d_{1})
\]
for any $(d_{1},d_{2})\in D^{2}$.

Any $\gamma\in\mathcal{A}^{2}G$ can canonically be identified with the mapping
$e\in D\mapsto\gamma_{e}^{1}\in\mathcal{A}G$, so that we can identify
$\mathcal{A}^{2}G$ and $(\mathcal{A}G)^{D}$. As is expected, this
identification enables us to define $\gamma_{2}\underset{1}{-}\gamma_{1}%
\in\mathcal{A}^{2}G$ for $\gamma_{1},\gamma_{2}\in\mathcal{A}^{2}G$, provided
that $\gamma_{1}(0,\cdot)=\gamma_{2}(0,\cdot)$. Similarly, we can define
$\gamma_{2}\underset{2}{-}\gamma_{1}\in\mathcal{A}^{2}G$ for $\gamma
_{1},\gamma_{2}\in\mathcal{A}^{2}G$, provided that $\gamma_{1}(\cdot
,0)=\gamma_{2}(\cdot,0)$. Given $\gamma_{1},\gamma_{2}\in\mathcal{A}^{2}G$,
their strong difference $\gamma_{2}\overset{\cdot}{-}\gamma_{1}\in
\mathcal{A}G$ is defined, provided that $\gamma_{1}\mid_{D(2)}=\gamma_{2}%
\mid_{D(2)}$. Lavendhomme's \cite{l1} treatment of strong difference
$\overset{\cdot}{-}$ in \S 3.4 carries over mutatis mutandis to our present
context. We note in passing the following simple proposition on strong
difference $\overset{\cdot}{-}$, which is not to be seen in our standard
reference \cite{l1} on synthetic differential geometry.

\begin{proposition}
\label{t1.5}For any $\gamma_{1},\gamma_{2},\gamma_{3}\in\mathcal{A}^{2}G$ with
$\gamma_{1}\mid_{D(2)}=\gamma_{2}\mid_{D(2)}=\gamma_{3}\mid_{D(2)} $, we have
\[
(\gamma_{2}\overset{\cdot}{-}\gamma_{1})+(\gamma_{3}\overset{\cdot}{-}%
\gamma_{2})+(\gamma_{1}\overset{\cdot}{-}\gamma_{3})=0
\]

\end{proposition}

\subsection{Differential Forms}

Given a groupoid $G$ and a vector bundle $E$ over the same space $M$, the
space $\mathbf{C}^{n}(G,E)$ of \textit{differential }$n$\textit{-forms with
values in} $E$ consists of all mappings $\omega$ from $\mathcal{A}^{n}G$ to
$E$ whose restriction to $\mathcal{A}_{x}^{n}G$ for each $x\in M$ takes values
in $E_{x}$ satisfying the following $n$-homogeneous and alternating properties:

\begin{enumerate}
\item We have
\[
\omega(a\underset{i}{\cdot}\gamma)=a\omega(\gamma)\text{ \ \ \ \ \ }(1\leq
i\leq n)
\]
for any $a\in\mathbb{R}$ and any $\gamma\in\mathcal{A}_{x}^{n}G$, where
$a\underset{i}{\cdot}\gamma\in\mathcal{A}_{x}^{n}G$ is defined to be
\[
(a\underset{i}{\cdot}\gamma)(d_{1},...,d_{n})=\gamma(d_{1},...,d_{i-1}%
,ad_{i},d_{i+1},...d_{n})
\]
for any $(d_{1},...,d_{n})\in D^{n}$.

\item We have
\[
\omega(\gamma\circ D^{\theta})=\mathrm{sign}(\theta)\omega(\gamma)
\]
for any permutation $\theta$ of $\{1,...,n\}$, where $D^{\theta}%
:D^{n}\rightarrow D^{n}$ permutes the $n$ coordinates by $\theta$.
\end{enumerate}

\section{Connection}

Let $\pi:H\rightarrow G$ be a morphism of groupoids over $M$. Let $L$ be the
kernel of $\pi$ with its canonical injection $\iota:L\rightarrow H$. It is
clear that $L$ is a group bundle over $M$. These entities shall be fixed
throughout the rest of the paper. Thus we have an exact sequence of groupoids
as follows:
\[
0\rightarrow L\overset{\iota}{\rightarrow}H\overset{\pi}{\rightarrow}G
\]
A \textit{connection} $\nabla$ \textit{with respect to} $\pi$ is a morphism
$\nabla:\mathcal{A}G\rightarrow\mathcal{A}H$ of vector bundles over $M$ such
that the composition $\pi_{\ast}\circ\nabla$ is the identity mapping of
$\mathcal{A}G$. A connection $\nabla$ with respect to $\pi$ shall be fixed
throughout the rest of the paper. If $G$ happens to be $M\times M$ (the pair
groupoid of $M$) with $\pi$ being the projection $h\in H\mapsto(\alpha
(h),\beta(h))\in M\times M$, our present notion of connection degenerates into
the classical one of infinitesimal connection.

Given $\gamma\in\mathcal{A}^{n+1}G$, we define $\gamma_{i}\in\mathcal{A}G $
($1\leq i\leq n+1$) to be
\[
\gamma_{i}(d)=\gamma(0,...,0,\underset{i}{d},0,...,0)
\]
for any $d\in D$. As in our previous paper \cite{n3}, we have

\begin{theorem}
Given $\omega\in\mathbf{C}^{n}(G,\mathcal{A}L)$, there exists a unique
$\mathbf{d}_{\nabla}\omega\in\mathbf{C}^{n+1}(G,\mathcal{A}L)$ such that
\begin{align*}
&  ((\mathbf{d}_{\nabla}\omega)(\gamma))_{d_{1}...d_{n+1}}\\
&  =\prod_{i=1}^{n+1}\{(\omega(\gamma_{0}^{i}))_{d_{1}...\widehat{d_{i}%
}...d_{n+1}}((\nabla\gamma_{i})_{d_{i}})^{-1}(\omega(\gamma_{d_{i}}%
^{i}))_{-d_{1}...\widehat{d_{i}}...d_{n+1}}(\nabla\gamma_{i})_{d_{i}%
}\}^{(-1)^{i}}%
\end{align*}
for any $\gamma\in\mathcal{A}^{n+1}G$ and any $(d_{1},...,d_{n+1})\in D^{n+1}$.
\end{theorem}

\section{A Lift of the Connection $\nabla$ to microsquares}

Let us define a mapping $\mathcal{A}^{2}G\rightarrow\mathcal{A}^{2}H$, which
shall be denoted by the same symbol $\nabla$ hopefully without any possible
confusion, to be
\[
\nabla\gamma(d_{1},d_{2})=(\nabla\gamma_{d_{1}}^{1})_{d_{2}}(\nabla\gamma
_{0}^{2})_{d_{1}}%
\]
for any $\gamma\in\mathcal{A}^{2}G$.

It is easy to see that

\begin{proposition}
\label{t3.1}For any $\gamma\in\mathcal{A}^{2}G$ and any $a\in\mathbb{R}$, we
have
\[%
\begin{array}
[c]{c}%
\nabla(a\underset{1}{\cdot}\gamma)=a\underset{1}{\cdot}\nabla\gamma\\
\nabla(a\underset{2}{\cdot}\gamma)=a\underset{2}{\cdot}\nabla\gamma
\end{array}
\]

\end{proposition}

\begin{corollary}
\label{t3.2}For any $\gamma_{1},\gamma_{2}\in\mathcal{A}^{2}G$, we have
\[%
\begin{array}
[c]{c}%
\nabla(\gamma_{2}\underset{1}{-}\gamma_{1})=\nabla\gamma_{2}\underset{1}%
{-}\nabla\gamma_{1}\text{ provided that }\gamma_{1}(0,\cdot)=\gamma
_{2}(0,\cdot)\text{;}\\
\nabla(\gamma_{2}\underset{2}{-}\gamma_{1})=\nabla\gamma_{2}\underset{2}%
{-}\nabla\gamma_{1}\text{ provided that }\gamma_{1}(\cdot,0)=\gamma_{2}%
(\cdot,0)\text{.}%
\end{array}
\]

\end{corollary}

\begin{proof}
This follows from the above proposition by Proposition 10 of \S 1.2 of
Lavendhomme \cite{l1}.
\end{proof}

\begin{proposition}
\label{t3.3}For any $t\in\mathcal{A}^{1}G$, we define $\varepsilon_{t}%
\in\mathcal{A}^{2}G$ to be
\[
\varepsilon_{t}(d_{1},d_{2})=t(d_{1}d_{2})
\]
Then we have
\[
(\nabla\varepsilon_{t})(d_{1},d_{2})=(\nabla t)(d_{1}d_{2})
\]
for any $d_{1},d_{2}\in D$.
\end{proposition}

\begin{proof}
It suffices to note that
\[
(\nabla\varepsilon_{t})(d_{1},d_{2})=(\nabla(d_{1}t))(d_{2})=(d_{1}\nabla
t)(d_{2})=(\nabla t)(d_{1}d_{2})
\]

\end{proof}

\begin{theorem}
\label{t3.4}For any $\gamma_{1},\gamma_{2}\in\mathcal{A}^{2}G$ with
$\gamma_{1}\mid_{D(2)}=\gamma_{2}\mid_{D(2)}$, we have
\[
\nabla(\gamma_{2}\overset{\cdot}{-}\gamma_{1})=\nabla\gamma_{2}\overset{\cdot
}{-}\nabla\gamma_{1}%
\]

\end{theorem}

\begin{proof}
Let $d_{1},d_{2}\in D$. We have
\begin{align*}
&  (\nabla(\gamma_{2}\overset{\cdot}{-}\gamma_{1}))(d_{1}d_{2})\\
&  =(\nabla\varepsilon_{\gamma_{2}\overset{\cdot}{-}\gamma_{1}})(d_{1}%
,d_{2})\text{ \ \ \ }\\
&  \text{[By Proposition \ref{t3.3}]}\\
&  =(\nabla((\gamma_{2}\underset{1}{-}\gamma_{1})\underset{2}{-}\tau
_{\gamma_{1}}^{2}))(d_{1},d_{2})\\
&  \text{[By Proposition 7 of \S 3.4 of Lavendhomme \cite{l1}]}\\
&  =((\nabla\gamma_{2}\underset{1}{-}\nabla\gamma_{1})\underset{2}{-}%
\nabla\tau_{\gamma_{1}}^{2})(d_{1},d_{2})\\
&  \text{[By Corollary \ref{t3.2}]}\\
&  =((\nabla\gamma_{2}\underset{1}{-}\nabla\gamma_{1})\underset{2}{-}%
\tau_{\nabla\gamma_{1}}^{2})(d_{1},d_{2})\\
&  =\varepsilon_{\nabla\gamma_{2}\overset{\cdot}{-}\nabla\gamma_{1}}%
(d_{1},d_{2})\\
&  \text{[By Proposition 7 of \S 3.4 of Lavendhomme \cite{l1}]}\\
&  =(\nabla\gamma_{2}\overset{\cdot}{-}\nabla\gamma_{1})(d_{1}d_{2})\\
&  \text{[By Proposition \ref{t3.3}]}%
\end{align*}
Since $d_{1},d_{2}\in D$ were arbitrary, the desired conclusion follows at once.
\end{proof}

\section{The Curvature Form}

\begin{proposition}
\label{t4.1}For any $\gamma\in\mathcal{A}^{2}G$, there exists a unique
$t\in\mathcal{A}^{1}L$ such that
\[
\iota(t_{d_{1}d_{2}})=((\nabla\gamma_{0}^{2})_{d_{1}})^{-1}((\nabla
\gamma_{d_{1}}^{1})_{d_{2}})^{-1}(\nabla\gamma_{d_{2}}^{2})_{d_{1}}%
(\nabla\gamma_{0}^{1})_{d_{2}}%
\]
for any $d_{1},d_{2}\in D$.
\end{proposition}

\begin{proof}
Let $\eta\in\mathcal{A}^{2}H$ to be
\[
\eta(d_{1},d_{2})=((\nabla\gamma_{0}^{2})_{d_{1}})^{-1}((\nabla\gamma_{d_{1}%
}^{1})_{d_{2}})^{-1}(\nabla\gamma_{d_{2}}^{2})_{d_{1}}(\nabla\gamma_{0}%
^{1})_{d_{2}}%
\]
for any $d_{1},d_{2}\in D$. Then it is easy to see that
\[
\eta(d,0)=\eta(0,d)=\mathrm{id}_{\alpha(\eta(0,0))}%
\]
Therefore there exists unique $t^{\prime}\in\mathcal{A}^{1}H$ such that
\[
t_{d_{1}d_{2}}^{\prime}=\eta(d_{1},d_{2})
\]
Furthermore we have
\begin{align*}
&  \pi(\eta(d_{1},d_{2}))\\
&  =\pi(((\nabla\gamma_{0}^{2})_{d_{1}})^{-1})\pi(((\nabla\gamma_{d_{1}}%
^{1})_{d_{2}})^{-1})\pi((\nabla\gamma_{d_{2}}^{2})_{d_{1}})\pi((\nabla
\gamma_{0}^{1})_{d_{2}})\\
&  =((\gamma_{0}^{2})_{d_{1}})^{-1}((\gamma_{d_{1}}^{1})_{d_{2}})^{-1}%
(\gamma_{d_{2}}^{2})_{d_{1}}(\gamma_{0}^{1})_{d_{2}}\\
&  =\gamma(d_{1},0)^{-1}(\gamma(d_{1},d_{2})\gamma(d_{1},0)^{-1})^{-1}%
\gamma(d_{1},d_{2})\gamma(0,d_{2})^{-1}\gamma(0,d_{2})\\
&  =\mathrm{id}_{\alpha(\eta(0,0))}%
\end{align*}
Therefore there exists a unique $t\in\mathcal{A}^{1}L$ with $\iota
(t)=t^{\prime}$. This completes the proof.
\end{proof}

We write $\Omega(\gamma)$ for the above $t$. Now we have

\begin{proposition}
\label{t4.2}The mapping $\Omega:\mathcal{A}^{2}G\rightarrow\mathcal{A}^{1}L $
consists in $\mathbf{C}^{2}(G,\mathcal{A}L)$.
\end{proposition}

\begin{proof}
We have to show that
\begin{align}
\Omega(a\underset{1}{\cdot}\gamma)  &  =a\Omega(\gamma)\label{3.1}\\
\Omega(a\underset{2}{\cdot}\gamma)  &  =a\Omega(\gamma)\label{3.2}\\
\Omega(\Sigma\gamma)  &  =-\Omega(\gamma) \label{3.3}%
\end{align}
for any $\gamma\in\mathcal{A}^{2}G$ and any $a\in\mathbb{R}$. Now we deal with
(\ref{3.1}), leaving a similar treatment of (\ref{3.2}) to the reader. Let
$d_{1},d_{2}\in D$. We have
\begin{align*}
&  \iota(\Omega(a\underset{1}{\cdot}\gamma))_{d_{1}d_{2}}\\
&  =((\nabla(a\underset{1}{\cdot}\gamma)_{0}^{2})_{d_{1}})^{-1}((\nabla
(a\underset{1}{\cdot}\gamma)_{d_{1}}^{1})_{d_{2}})^{-1}(\nabla(a\underset
{1}{\cdot}\gamma)_{d_{2}}^{2})_{d_{1}}(\nabla(a\underset{1}{\cdot}\gamma
)_{0}^{1})_{d_{2}}\\
&  =((\nabla\gamma_{0}^{2})_{ad_{1}})^{-1}((\nabla\gamma_{ad_{1}}^{1})_{d_{2}%
})^{-1}(\nabla\gamma_{d_{2}}^{2})_{ad_{1}}(\nabla\gamma_{0}^{1})_{d_{2}}\\
&  =\iota(\Omega(\gamma))_{ad_{1}d_{2}}\\
&  =\iota(a\Omega(\gamma))_{d_{1}d_{2}}%
\end{align*}
Now we deal with (\ref{3.3}). We have
\begin{align*}
&  \iota(\Omega(\Sigma\gamma))\iota(\Omega(\gamma))_{d_{1}d_{2}}\\
&  =\{((\nabla\gamma_{0}^{1})_{d_{2}})^{-1}((\nabla\gamma_{d_{2}}^{2})_{d_{1}%
})^{-1}(\nabla\gamma_{d_{1}}^{1})_{d_{2}}(\nabla\gamma_{0}^{2})_{d_{1}%
}\}\{((\nabla\gamma_{0}^{2})_{d_{1}})^{-1}((\nabla\gamma_{d_{1}}^{1})_{d_{2}%
})^{-1}(\nabla\gamma_{d_{2}}^{2})_{d_{1}}(\nabla\gamma_{0}^{1})_{d_{2}}\}\\
&  =\mathrm{id}_{\alpha(\gamma(0,0))}%
\end{align*}
Thie completes the proof.
\end{proof}

We call $\Omega$ the \textit{curvature form} of $\nabla$.

\begin{proposition}
\label{t4.5}For any $\gamma\in\mathcal{A}^{2}G$, we have
\[
\Omega(\gamma)=\Sigma\nabla\Sigma\gamma\overset{\cdot}{-}\nabla\gamma
\]

\end{proposition}

\begin{proof}
As in the proof of Proposition 8 of \S 3.4 of Lavendhomme \cite{l1}, let us
consider a function $\mathbf{l}:D^{2}\vee D\rightarrow H$ given by
\[
\mathbf{l}(d_{1},d_{2},e)=(\nabla\gamma_{d_{1}}^{1})_{d_{2}}(\nabla\gamma
_{0}^{2})_{d_{1}}\Omega(\gamma)_{e}%
\]
for any $(d_{1},d_{2},e)\in D^{2}\vee D$. Then it is easy to see that
$\mathbf{l}(d_{1},d_{2},0)=(\nabla\gamma)(d_{1},d_{2})$ and $\mathbf{l}%
(d_{1},d_{2},d_{1}d_{2})=(\Sigma\nabla\Sigma\gamma)(d_{1},d_{2})$. Therefore
we have
\begin{align*}
&  (\Sigma\nabla\Sigma\gamma\overset{\cdot}{-}\nabla\gamma)_{e}\\
&  =\mathbf{l}(0,0,e)\\
&  =\Omega(\gamma)_{e}%
\end{align*}
This completes the proof.
\end{proof}

Now we deal with tensorial aspects of $\Omega$. It is easy to see that

\begin{proposition}
\label{t4.3}Let $X,Y\in\Gamma(\mathcal{A}G)$. Then we have
\[
\nabla(Y\ast X)=\nabla Y\ast\nabla X
\]

\end{proposition}

Now we have the following familiar form for $\Omega$.

\begin{theorem}
\label{t4.4}Let $X,Y\in\Gamma(\mathcal{A}G)$. Then we have
\[
\Omega(Y\ast X)=\nabla\lbrack X,Y]-[\nabla X,\nabla Y]
\]

\end{theorem}

\begin{proof}
It suffices to note that
\begin{align*}
&  \Omega(Y\ast X)\\
&  =\Sigma\nabla\Sigma(Y\ast X)\overset{\cdot}{-}\nabla(Y\ast X)\\
&  \text{[By Proposition \ref{t4.5}]}\\
&  =\nabla\Sigma(Y\ast X)\overset{\cdot}{-}\Sigma\nabla(Y\ast X)\\
&  \text{[By Proposition 6 of \S 3.4 of Lavendhomme \cite{l1}]}\\
&  =\nabla(\Sigma(Y\ast X)\overset{\cdot}{-}X\ast Y)-(\Sigma\nabla(Y\ast
X)\overset{\cdot}{-}\nabla(X\ast Y))\\
&  \text{[By Proposition \ref{t1.5}]}\\
&  =\nabla(Y\ast X\overset{\cdot}{-}\Sigma(X\ast Y))-(\nabla(Y\ast
X)\overset{\cdot}{-}\Sigma\nabla(X\ast Y))\\
&  \text{[By Proposition 6 of \S 3.4 of Lavendhomme \cite{l1}]}\\
&  =\nabla(Y\ast X\overset{\cdot}{-}\Sigma(X\ast Y))-(\nabla Y\ast\nabla
X\overset{\cdot}{-}\Sigma(\nabla X\ast\nabla Y))\\
&  \text{[By Proposition \ref{t4.3}]}\\
&  =\nabla\lbrack X,Y]-[\nabla X,\nabla Y]\\
&  \text{[By Proposition 8 of \S 3.4 of Lavendhomme \cite{l1}]}%
\end{align*}

\end{proof}

\section{The Bianchi Identity}

Let us begin with the following abstract Bianchi identity, which traces back
to Kock \cite{k4}, though our version is cubical, while Kock's one is simplicial.

\begin{theorem}
Let $\gamma\in\mathcal{A}^{2}G$. Let $d_{1},d_{2},d_{3}\in D$. We denote
points $\beta(\gamma(0,0,0))$, $\beta(\gamma(d_{1},0,0))$, $\beta
(\gamma(0,d_{2},0))$, $\beta(\gamma(0,0,d_{3}))$, $\beta(\gamma(d_{1}%
,d_{2},0))$, $\beta(\gamma(d_{1},0,d_{3}))$, $\beta(\gamma(0,d_{2},d_{3}))$
and $\beta(\gamma(d_{1},d_{2},d_{3}))$ by $O$, $A$, $B$, $C$, $D$, $E$, $F$
and $G$ respectively. These eight points are depicted figuratively as the
eight vertices of a cube:
\end{theorem}

$%
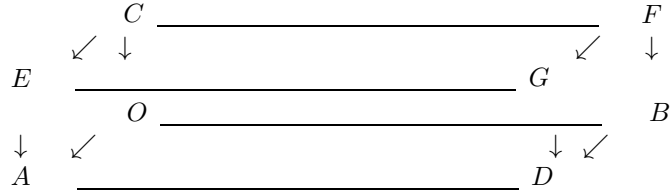
\begin{array}
[c]{ccc}
& \ \ \ \
\begin{array}
[c]{c}%
C
\end{array}
\underline
{\ \ \ \ \ \ \ \ \ \ \ \ \ \ \ \ \ \ \ \ \ \ \ \ \ \ \ \ \ \ \ \ \ \ \ \ \ \ \ \ \ \ \ \ \ \ \ \ \ \ }%
\begin{array}
[c]{c}
\end{array}
& F\\
&
\begin{array}
[c]{c}%
\swarrow
\end{array}
\downarrow
\ \ \ \ \ \ \ \ \ \ \ \ \ \ \ \ \ \ \ \ \ \ \ \ \ \ \ \ \ \ \ \ \ \ \ \ \ \ \ \ \ \ \ \ \ \ \ \
\begin{array}
[c]{c}%
\swarrow
\end{array}
& \downarrow\\
E &
\begin{array}
[c]{c}
\end{array}
\underline
{\ \ \ \ \ \ \ \ \ \ \ \ \ \ \ \ \ \ \ \ \ \ \ \ \ \ \ \ \ \ \ \ \ \ \ \ \ \ \ \ \ \ \ \ \ \ \ \ \ \ }%
\begin{array}
[c]{c}%
G
\end{array}
\ \ \ \  &
\end{array}
$

$%
\begin{array}
[c]{ccc}
& \ \ \ \
\begin{array}
[c]{c}%
O
\end{array}
\underline
{\ \ \ \ \ \ \ \ \ \ \ \ \ \ \ \ \ \ \ \ \ \ \ \ \ \ \ \ \ \ \ \ \ \ \ \ \ \ \ \ \ \ \ \ \ \ \ \ \ \ }%
\begin{array}
[c]{c}
\end{array}
& B\\
\downarrow &
\begin{array}
[c]{c}%
\swarrow
\end{array}
\ \ \ \ \ \ \ \ \ \ \ \ \ \ \ \ \ \ \ \ \ \ \ \ \ \ \ \ \ \ \ \ \ \ \ \ \ \ \ \ \ \ \ \ \ \ \ \ \ \downarrow%
\begin{array}
[c]{c}%
\swarrow
\end{array}
& \\
A &
\begin{array}
[c]{c}
\end{array}
\underline
{\ \ \ \ \ \ \ \ \ \ \ \ \ \ \ \ \ \ \ \ \ \ \ \ \ \ \ \ \ \ \ \ \ \ \ \ \ \ \ \ \ \ \ \ \ \ \ \ \ \ }%
\begin{array}
[c]{c}%
D
\end{array}
\ \ \ \  &
\end{array}
$

\begin{theorem}
For each pair $(X,Y)$ of adjacent vertices $X,Y$ of the cube, $P_{XY}$ denotes
the following arrow in $H$, while $P_{YX}$ denotes the inverse of $P_{XY}$:
\begin{align*}
P_{OA} &  =(\nabla\gamma_{0,0}^{2,3})_{d_{1}}\\
P_{OB} &  =(\nabla\gamma_{0,0}^{1,3})_{d_{2}}\\
P_{OC} &  =(\nabla\gamma_{0,0}^{1,2})_{d_{3}}\\
P_{AD} &  =(\nabla\gamma_{d_{1},0}^{1,3})_{d_{2}}\\
P_{AE} &  =(\nabla\gamma_{d_{1},0}^{1,2})_{d_{3}}\\
P_{BD} &  =(\nabla\gamma_{d_{2},0}^{2,3})_{d_{1}}\\
P_{BF} &  =(\nabla\gamma_{0,d_{2}}^{1,2})_{d_{3}}\\
P_{CE} &  =(\nabla\gamma_{0,d_{3}}^{2,3})_{d_{1}}\\
P_{CF} &  =(\nabla\gamma_{0,d_{3}}^{1,3})_{d_{2}}\\
P_{DG} &  =(\nabla\gamma_{d_{1},d_{2}}^{1,2})_{d_{3}}\\
P_{EG} &  =(\nabla\gamma_{d_{2},d_{3}}^{2,3})_{d_{1}}\\
P_{FG} &  =(\nabla\gamma_{d_{1},d_{3}}^{1,3})_{d_{2}}%
\end{align*}
For any four vertices $X,Y,Z,W$ of the cube rounding one of the six facial
squares of the cube, $R_{XYZW}$ denotes $P_{WX}P_{ZW}P_{YZ}P_{XY}$. Then we
have
\begin{align*}
&  P_{AO}P_{DA}P_{GD}R_{GFBD}R_{GECF}R_{GDAE}P_{DG}P_{AD}P_{OA}R_{OCEA}%
R_{OBFC}R_{OADB}\\
&  =\mathrm{id}_{O}%
\end{align*}

\end{theorem}

\begin{proof}
Write over the desired identity exclusively in terms of $P_{XY}$'s, and write
off all consective $P_{XY}P_{YX}$'s.
\end{proof}

Now we are ready to establish the second Bianchi identity in familiar form.

\begin{theorem}
We have
\[
\mathbf{d}_{\nabla}\Omega=0
\]

\end{theorem}

\begin{proof}
Let $\gamma,d_{1},d_{2},d_{3},O,A,B,C,D,E,F,G$ be the same as in the previous
theorem. By the very definition of $\Omega$, we have
\[
R_{OADB}=\Omega(\gamma_{0}^{3})_{-d_{1}d_{2}}%
\]%
\[
R_{OBFC}=\Omega(\gamma_{0}^{1})_{-d_{2}d_{3}}%
\]%
\[
R_{OCEA}=\Omega(\gamma_{0}^{2})_{d_{1}d_{3}}%
\]
Now we have the following three calculations:
\begin{align*}
&  P_{AO}P_{DA}P_{GD}R_{GDAE}P_{DG}P_{AD}P_{OA}\\
&  =P_{AO}R_{AEGD}P_{OA}\\
&  =((\nabla\gamma_{0,0}^{2,3})_{d_{1}})^{-1}\Omega(\gamma_{d_{1}}^{1}%
)_{d_{2}d_{3}}(\nabla\gamma_{0,0}^{2,3})_{d_{1}}%
\end{align*}%
\begin{align*}
&  P_{AO}P_{DA}P_{GD}R_{GECF}P_{DG}P_{AD}P_{OA}\\
&  =P_{AO}P_{DA}P_{GD}P_{EG}P_{CE}R_{CFGE}P_{EC}P_{GE}P_{DG}P_{AD}P_{OA}\\
&  =P_{AO}R_{AEGD}P_{EA}P_{CE}R_{CFGE}P_{EC}P_{AE}R_{ADGE}P_{OA}\\
&  =R_{OCEA}P_{CO}P_{EC}P_{AE}R_{AEGD}P_{EA}P_{CE}R_{CFGE}P_{EC}P_{AE}%
R_{ADGE}P_{EA}P_{CE}P_{OC}R_{OAEC}\\
&  =\Omega(\gamma_{0}^{2})_{d_{1}d_{3}}((\nabla\gamma_{0,0}^{1,2})_{d_{3}%
})^{-1}\{((\nabla\gamma_{0,d_{3}}^{2,3})_{d_{1}})^{-1}(\nabla\gamma_{d_{1}%
,0}^{1,2})_{d_{3}}\Omega(\gamma_{d_{1}}^{1})_{d_{2}d_{3}}((\nabla\gamma
_{d_{1},0}^{1,2})_{d_{3}})^{-1}(\nabla\gamma_{0,d_{3}}^{2,3})_{d_{1}}\}\\
&  \Omega(\gamma_{d_{3}}^{3})_{d_{1}d_{2}}\{((\nabla\gamma_{0,d_{3}}%
^{2,3})_{d_{1}})^{-1}(\nabla\gamma_{d_{1},0}^{1,2})_{d_{3}}\Omega
(\gamma_{d_{1}}^{1})_{-d_{2}d_{3}}((\nabla\gamma_{d_{1},0}^{1,2})_{d_{3}%
})^{-1}(\nabla\gamma_{0,d_{3}}^{2,3})_{d_{1}}\}\\
&  (\nabla\gamma_{0,0}^{1,2})_{d_{3}}\Omega(\gamma_{0}^{2})_{-d_{1}d_{3}}\\
&  =\Omega(\gamma_{0}^{2})_{d_{1}d_{3}}\{((\nabla\gamma_{0,0}^{1,2})_{d_{3}%
})^{-1}\Omega(\gamma_{d_{3}}^{3})_{d_{1}d_{2}}(\nabla\gamma_{0,0}%
^{1,2})_{d_{3}}\}\Omega(\gamma_{0}^{2})_{-d_{1}d_{3}}\\
&  \text{[By Proposition \ref{t1.2}, cf. Figure (\ref{Figure 1})]}\\
&  =((\nabla\gamma_{0,0}^{1,2})_{d_{3}})^{-1}\Omega(\gamma_{d_{3}}^{3}%
)_{d_{1}d_{2}}(\nabla\gamma_{0,0}^{1,2})_{d_{3}}\\
&  \text{[By Proposition \ref{t1.2}, cf. Figure (\ref{Figure 2})]}%
\end{align*}%
\begin{equation}%
\begin{array}
[c]{ccc}%
C & \underrightarrow{((\nabla\gamma_{0,d_{3}}^{2,3})_{d_{1}})^{-1}%
(\nabla\gamma_{d_{1},0}^{1,2})_{d_{3}}\Omega(\gamma_{d_{1}}^{1})_{-d_{2}d_{3}%
}((\nabla\gamma_{d_{1},0}^{1,2})_{d_{3}})^{-1}(\nabla\gamma_{0,d_{3}}%
^{2,3})_{d_{1}}} & C\\%
\begin{array}
[c]{cc}%
\Omega(\gamma_{d_{3}}^{3})_{d_{1}d_{2}} & \downarrow
\end{array}
& \circlearrowleft &
\begin{array}
[c]{cc}%
\downarrow & \Omega(\gamma_{d_{3}}^{3})_{d_{1}d_{2}}%
\end{array}
\\
C & \overrightarrow{((\nabla\gamma_{0,d_{3}}^{2,3})_{d_{1}})^{-1}(\nabla
\gamma_{d_{1},0}^{1,2})_{d_{3}}\Omega(\gamma_{d_{1}}^{1})_{-d_{2}d_{3}%
}((\nabla\gamma_{d_{1},0}^{1,2})_{d_{3}})^{-1}(\nabla\gamma_{0,d_{3}}%
^{2,3})_{d_{1}}} & C
\end{array}
\label{Figure 1}%
\end{equation}%
\begin{equation}%
\begin{array}
[c]{ccc}%
O & \underrightarrow{((\nabla\gamma_{0,0}^{1,2})_{d_{3}})^{-1}\Omega
(\gamma_{d_{3}}^{3})_{d_{1}d_{2}}(\nabla\gamma_{0,0}^{1,2})_{d_{3}}} & O\\%
\begin{array}
[c]{cc}%
\Omega(\gamma_{0}^{2})_{-d_{1}d_{3}} & \downarrow
\end{array}
& \circlearrowleft &
\begin{array}
[c]{cc}%
\downarrow & \Omega(\gamma_{0}^{2})_{-d_{1}d_{3}}%
\end{array}
\\
O & \overrightarrow{((\nabla\gamma_{0,0}^{1,2})_{d_{3}})^{-1}\Omega
(\gamma_{d_{3}}^{3})_{d_{1}d_{2}}(\nabla\gamma_{0,0}^{1,2})_{d_{3}}} & O
\end{array}
\label{Figure 2}%
\end{equation}%
\begin{align*}
&  P_{AO}P_{DA}P_{GD}R_{GFBD}P_{DG}P_{AD}P_{OA}\\
&  =P_{AO}P_{DA}R_{DGFB}P_{AD}P_{OA}\\
&  =R_{OBDA}P_{BO}R_{BDGF}P_{OB}R_{OADB}\\
&  =\Omega(\gamma_{0}^{3})_{d_{1}d_{2}}\{((\nabla\gamma_{0,0}^{1,3})_{d_{2}%
})^{-1}\Omega(\gamma_{d_{2}}^{2})_{-d_{1}d_{3}}(\nabla\gamma_{0,0}%
^{1,3})_{d_{2}}\}\Omega(\gamma_{0}^{3})_{-d_{1}d_{2}}\\
&  =((\nabla\gamma_{0,0}^{1,3})_{d_{2}})^{-1}\Omega(\gamma_{d_{2}}%
^{2})_{-d_{1}d_{3}}(\nabla\gamma_{0,0}^{1,3})_{d_{2}}\\
&  \text{[By Proposition \ref{t1.2}, cf. Figure (\ref{Figure 3})]}%
\end{align*}%
\begin{equation}%
\begin{array}
[c]{ccc}%
O & \underrightarrow{((\nabla\gamma_{0,0}^{1,3})_{d_{2}})^{-1}\Omega
(\gamma_{d_{2}}^{2})_{-d_{1}d_{3}}(\nabla\gamma_{0,0}^{1,3})_{d_{2}}} & O\\%
\begin{array}
[c]{cc}%
\Omega(\gamma_{0}^{3})_{-d_{1}d_{2}} & \downarrow
\end{array}
& \circlearrowleft &
\begin{array}
[c]{cc}%
\downarrow & \Omega(\gamma_{0}^{3})_{-d_{1}d_{2}}%
\end{array}
\\
O & \overrightarrow{((\nabla\gamma_{0,0}^{1,3})_{d_{2}})^{-1}\Omega
(\gamma_{d_{2}}^{2})_{-d_{1}d_{3}}(\nabla\gamma_{0,0}^{1,3})_{d_{2}}} & O
\end{array}
\label{Figure 3}%
\end{equation}
Therefore we have
\begin{align*}
&  (\mathbf{d}_{\nabla}\Omega(\gamma))_{-d_{1}d_{2}d_{3}}\\
&  =\Omega(\gamma_{0}^{1})_{-d_{2}d_{3}}\{((\nabla\gamma_{0,0}^{2,3})_{d_{1}%
})^{-1}\Omega(\gamma_{d_{1}}^{1})_{d_{2}d_{3}}(\nabla\gamma_{0,0}%
^{2,3})_{d_{1}}\}\\
&  \Omega(\gamma_{0}^{2})_{d_{1}d_{3}}\{((\nabla\gamma_{0,0}^{1,3})_{d_{2}%
})^{-1}\Omega(\gamma_{d_{2}}^{2})_{-d_{1}d_{3}}(\nabla\gamma_{0,0}%
^{1,3})_{d_{2}}\}\\
&  \Omega(\gamma_{0}^{3})_{-d_{1}d_{2}}\{((\nabla\gamma_{0,0}^{1,2})_{d_{3}%
})^{-1}\Omega(\gamma_{d_{3}}^{3})_{d_{1}d_{2}}(\nabla\gamma_{0,0}%
^{1,2})_{d_{3}}\}\\
&  =\{((\nabla\gamma_{0,0}^{2,3})_{d_{1}})^{-1}\Omega(\gamma_{d_{1}}%
^{1})_{d_{2}d_{3}}(\nabla\gamma_{0,0}^{2,3})_{d_{1}}\}\{((\nabla\gamma
_{0,0}^{1,3})_{d_{2}})^{-1}\Omega(\gamma_{d_{2}}^{2})_{-d_{1}d_{3}}%
(\nabla\gamma_{0,0}^{1,3})_{d_{2}}\}\\
&  \{((\nabla\gamma_{0,0}^{1,2})_{d_{3}})^{-1}\Omega(\gamma_{d_{3}}%
^{3})_{d_{1}d_{2}}(\nabla\gamma_{0,0}^{1,2})_{d_{3}}\}\Omega(\gamma_{0}%
^{1})_{-d_{2}d_{3}}\Omega(\gamma_{0}^{2})_{d_{1}d_{3}}\Omega(\gamma_{0}%
^{3})_{-d_{1}d_{2}}\\
&  \text{[By Proposition \ref{t1.2}, cf. Figures (\ref{Figure 4}%
)-(\ref{Figure 9})]}\\
&  =\{((\nabla\gamma_{0,0}^{1,3})_{d_{2}})^{-1}\Omega(\gamma_{d_{2}}%
^{2})_{-d_{1}d_{3}}(\nabla\gamma_{0,0}^{1,3})_{d_{2}}\}\{((\nabla\gamma
_{0,0}^{1,2})_{d_{3}})^{-1}\Omega(\gamma_{d_{3}}^{3})_{d_{1}d_{2}}%
(\nabla\gamma_{0,0}^{1,2})_{d_{3}}\}\\
&  \{((\nabla\gamma_{0,0}^{2,3})_{d_{1}})^{-1}\Omega(\gamma_{d_{1}}%
^{1})_{d_{2}d_{3}}(\nabla\gamma_{0,0}^{2,3})_{d_{1}}\}\Omega(\gamma_{0}%
^{2})_{d_{1}d_{3}}\Omega(\gamma_{0}^{1})_{-d_{2}d_{3}}\Omega(\gamma_{0}%
^{3})_{-d_{1}d_{2}}\\
&  \text{[By Proposition \ref{t1.2}, cf. Figures (\ref{Figure 10}%
)-(\ref{Figure 12})]}\\
&  =\{P_{AO}P_{DA}P_{GD}R_{GFBD}P_{DG}P_{AD}P_{OA}\}\{P_{AO}P_{DA}%
P_{GD}R_{GECF}P_{DG}P_{AD}P_{OA}\}\\
&  \{P_{AO}P_{DA}P_{GD}R_{GDAE}P_{DG}P_{AD}P_{OA}\}R_{OCEA}R_{OBFC}R_{OADB}\\
&  =\mathrm{id}_{O}%
\end{align*}%
\begin{align}
&
\begin{array}
[c]{ccc}%
O & \underrightarrow{((\nabla\gamma_{0,0}^{1,2})_{d_{3}})^{-1}\Omega
(\gamma_{d_{3}}^{3})_{d_{1}d_{2}}(\nabla\gamma_{0,0}^{1,2})_{d_{3}}} & O\\%
\begin{array}
[c]{cc}%
\Omega(\gamma_{0}^{1})_{-d_{2}d_{3}} & \downarrow
\end{array}
& \circlearrowleft &
\begin{array}
[c]{cc}%
\downarrow & \Omega(\gamma_{0}^{1})_{-d_{2}d_{3}}%
\end{array}
\\
O & \overrightarrow{((\nabla\gamma_{0,0}^{1,2})_{d_{3}})^{-1}\Omega
(\gamma_{d_{3}}^{3})_{d_{1}d_{2}}(\nabla\gamma_{0,0}^{1,2})_{d_{3}}} & O
\end{array}
\label{Figure 4}\\
&
\begin{array}
[c]{ccc}%
O & \underrightarrow{(\nabla\gamma_{0,0}^{1,3})_{d_{2}})^{-1}\Omega
(\gamma_{d_{2}}^{2})_{-d_{1}d_{3}}(\nabla\gamma_{0,0}^{1,3})_{d_{2}}} & O\\%
\begin{array}
[c]{cc}%
\Omega(\gamma_{0}^{1})_{-d_{2}d_{3}} & \downarrow
\end{array}
& \circlearrowleft &
\begin{array}
[c]{cc}%
\downarrow & \Omega(\gamma_{0}^{1})_{-d_{2}d_{3}}%
\end{array}
\\
O & \overrightarrow{(\nabla\gamma_{0,0}^{1,3})_{d_{2}})^{-1}\Omega
(\gamma_{d_{2}}^{2})_{-d_{1}d_{3}}(\nabla\gamma_{0,0}^{1,3})_{d_{2}}} & O
\end{array}
\label{Figure 5}\\
&
\begin{array}
[c]{ccc}%
O & \underrightarrow{((\nabla\gamma_{0,0}^{1,2})_{d_{3}})^{-1}\Omega
(\gamma_{d_{3}}^{3})_{d_{1}d_{2}}(\nabla\gamma_{0,0}^{1,2})_{d_{3}}} & O\\%
\begin{array}
[c]{cc}%
\Omega(\gamma_{0}^{1})_{-d_{2}d_{3}} & \downarrow
\end{array}
& \circlearrowleft &
\begin{array}
[c]{cc}%
\downarrow & \Omega(\gamma_{0}^{1})_{-d_{2}d_{3}}%
\end{array}
\\
O & \overrightarrow{((\nabla\gamma_{0,0}^{1,2})_{d_{3}})^{-1}\Omega
(\gamma_{d_{3}}^{3})_{d_{1}d_{2}}(\nabla\gamma_{0,0}^{1,2})_{d_{3}}} & O
\end{array}
\label{Figure 6}\\
&
\begin{array}
[c]{ccc}%
O & \underrightarrow{((\nabla\gamma_{0,0}^{1,3})_{d_{2}})^{-1}\Omega
(\gamma_{d_{2}}^{2})_{-d_{1}d_{3}}(\nabla\gamma_{0,0}^{1,3})_{d_{2}}} & O\\%
\begin{array}
[c]{cc}%
\Omega(\gamma_{0}^{2})_{d_{1}d_{3}} & \downarrow
\end{array}
& \circlearrowleft &
\begin{array}
[c]{cc}%
\downarrow & \Omega(\gamma_{0}^{2})_{d_{1}d_{3}}%
\end{array}
\\
O & \overrightarrow{((\nabla\gamma_{0,0}^{1,3})_{d_{2}})^{-1}\Omega
(\gamma_{d_{2}}^{2})_{-d_{1}d_{3}}(\nabla\gamma_{0,0}^{1,3})_{d_{2}}} & O
\end{array}
\label{Figure 7}\\
&
\begin{array}
[c]{ccc}%
O & \underrightarrow{((\nabla\gamma_{0,0}^{1,2})_{d_{3}})^{-1}\Omega
(\gamma_{d_{3}}^{3})_{d_{1}d_{2}}(\nabla\gamma_{0,0}^{1,2})_{d_{3}}} & O\\%
\begin{array}
[c]{cc}%
\Omega(\gamma_{0}^{2})_{d_{1}d_{3}} & \downarrow
\end{array}
& \circlearrowleft &
\begin{array}
[c]{cc}%
\downarrow & \Omega(\gamma_{0}^{2})_{d_{1}d_{3}}%
\end{array}
\\
O & \overrightarrow{((\nabla\gamma_{0,0}^{1,2})_{d_{3}})^{-1}\Omega
(\gamma_{d_{3}}^{3})_{d_{1}d_{2}}(\nabla\gamma_{0,0}^{1,2})_{d_{3}}} & O
\end{array}
\label{Figure 8}\\
&
\begin{array}
[c]{ccc}%
O & \underrightarrow{((\nabla\gamma_{0,0}^{1,2})_{d_{3}})^{-1}\Omega
(\gamma_{d_{3}}^{3})_{d_{1}d_{2}}(\nabla\gamma_{0,0}^{1,2})_{d_{3}}} & O\\%
\begin{array}
[c]{cc}%
\Omega(\gamma_{0}^{3})_{-d_{1}d_{2}} & \downarrow
\end{array}
& \circlearrowleft &
\begin{array}
[c]{cc}%
\downarrow & \Omega(\gamma_{0}^{3})_{-d_{1}d_{2}}%
\end{array}
\\
O & \overrightarrow{((\nabla\gamma_{0,0}^{1,2})_{d_{3}})^{-1}\Omega
(\gamma_{d_{3}}^{3})_{d_{1}d_{2}}(\nabla\gamma_{0,0}^{1,2})_{d_{3}}} & O
\end{array}
\label{Figure 9}%
\end{align}%
\begin{align}
&
\begin{array}
[c]{ccc}%
O & \underrightarrow{{\tiny ((\nabla\gamma}_{0,0}^{2,3}{\tiny )}_{d_{1}%
}{\tiny )}^{-1}{\tiny \Omega(\gamma}_{d_{1}}^{1}{\tiny )}_{d_{2}d_{3}%
}{\tiny (\nabla\gamma}_{0,0}^{2,3}{\tiny )}_{d_{1}}} & O\\%
\begin{array}
[c]{cc}%
{\tiny ((\nabla\gamma}_{0,0}^{1,3}{\tiny )}_{d_{2}}{\tiny )}^{-1}%
{\tiny \Omega(\gamma}_{d_{2}}^{2}{\tiny )}_{-d_{1}d_{3}}{\tiny (\nabla\gamma
}_{0,0}^{1,3}{\tiny )}_{d_{2}} & \downarrow
\end{array}
& \circlearrowleft &
\begin{array}
[c]{cc}%
\downarrow & {\tiny ((\nabla\gamma}_{0,0}^{1,3}{\tiny )}_{d_{2}}{\tiny )}%
^{-1}{\tiny \Omega(\gamma}_{d_{2}}^{2}{\tiny )}_{-d_{1}d_{3}}{\tiny (\nabla
\gamma}_{0,0}^{1,3}{\tiny )}_{d_{2}}%
\end{array}
\\
O & \overrightarrow{{\tiny ((\nabla\gamma}_{0,0}^{2,3}{\tiny )}_{d_{1}%
}{\tiny )}^{-1}{\tiny \Omega(\gamma}_{d_{1}}^{1}{\tiny )}_{d_{2}d_{3}%
}{\tiny (\nabla\gamma}_{0,0}^{2,3}{\tiny )}_{d_{1}}} & O
\end{array}
\label{Figure 10}\\
&
\begin{array}
[c]{ccc}%
O & \underrightarrow{{\tiny ((\nabla\gamma}_{0,0}^{2,3}{\tiny )}_{d_{1}%
}{\tiny )}^{-1}{\tiny \Omega(\gamma}_{d_{1}}^{1}{\tiny )}_{d_{2}d_{3}%
}{\tiny (\nabla\gamma}_{0,0}^{2,3}{\tiny )}_{d_{1}}} & O\\%
\begin{array}
[c]{cc}%
{\tiny ((\nabla\gamma}_{0,0}^{1,2}{\tiny )}_{d_{3}}{\tiny )}^{-1}%
{\tiny \Omega(\gamma}_{d_{3}}^{3}{\tiny )}_{d_{1}d_{2}}{\tiny (\nabla\gamma
}_{0,0}^{1,2}{\tiny )}_{d_{3}} & \downarrow
\end{array}
& \circlearrowleft &
\begin{array}
[c]{cc}%
\downarrow & {\tiny ((\nabla\gamma}_{0,0}^{1,2}{\tiny )}_{d_{3}}{\tiny )}%
^{-1}{\tiny \Omega(\gamma}_{d_{3}}^{3}{\tiny )}_{d_{1}d_{2}}{\tiny (\nabla
\gamma}_{0,0}^{1,2}{\tiny )}_{d_{3}}%
\end{array}
\\
O & \overrightarrow{{\tiny ((\nabla\gamma}_{0,0}^{2,3}{\tiny )}_{d_{1}%
}{\tiny )}^{-1}{\tiny \Omega(\gamma}_{d_{1}}^{1}{\tiny )}_{d_{2}d_{3}%
}{\tiny (\nabla\gamma}_{0,0}^{2,3}{\tiny )}_{d_{1}}} & O
\end{array}
\label{Fiure 11}%
\end{align}%
\begin{equation}%
\begin{array}
[c]{ccc}%
O & \underrightarrow{\Omega(\gamma_{0}^{1})_{-d_{2}d_{3}}} & O\\%
\begin{array}
[c]{cc}%
\Omega(\gamma_{0}^{2})_{d_{1}d_{3}} & \downarrow
\end{array}
& \circlearrowleft &
\begin{array}
[c]{cc}%
\downarrow & \Omega(\gamma_{0}^{2})_{d_{1}d_{3}}%
\end{array}
\\
O & \overrightarrow{\Omega(\gamma_{0}^{1})_{-d_{2}d_{3}}} & O
\end{array}
\label{Figure 12}%
\end{equation}
This completes the proof.
\end{proof}

\end{document}